\documentclass[12pt, a4paper]{article}

\usepackage{hyperref}

\usepackage{graphicx} % Required for inserting images
\usepackage[sumlimits]{amsmath}
\usepackage{amssymb,amsfonts,amsthm}
\usepackage{color,xcolor}
\usepackage{todonotes}
\usepackage{url}

\newtheorem{thm}{Theorem}
\newtheorem{cor}[thm]{Corollary}
\newtheorem{lem}[thm]{Lemma}
\newtheorem{prop}[thm]{Proposition}
\newtheorem{defin}[thm]{Definition}

\newtheorem{rem}[thm]{Remark}

\def\P{\mathcal{P}}
\def\S{\mathcal{S}}

\def\D{\mathbb{D}}

\def\1{\mathbbm{1}}

\newcommand{\ld}{{\rm ld}\,}

\newcommand{\wce}{{\rm wce}}
\newcommand{\uu}{\mathfrak{u}}
\newcommand{\vv}{\mathfrak{v}}

\allowdisplaybreaks

\newcommand{\nicolas}[1]{\begingroup\color{teal}#1\endgroup}
\newcommand{\peter}[1]{\begingroup\color{violet}#1\endgroup}

\usepackage{mathtools}

\newcommand{\N}{{\mathbb{N}}} % natural numbers {1, 2, ...}
\newcommand{\bsx}{{\boldsymbol{x}}}
\newcommand{\bsy}{{\boldsymbol{y}}}
\newcommand{\bsj}{{\boldsymbol{j}}}
\newcommand{\bsk}{{\boldsymbol{k}}}
\newcommand{\rd}{\,\mathrm{d}} % differential symbol with tiny space in front for use in integrals
\newcommand{\bsh}{{\boldsymbol{h}}}
\newcommand{\ZZ}{{\mathbb{Z}}} % integers
\newcommand{\bszero}{{\boldsymbol{0}}} % vector of zeros
\newcommand{\RR}{{\mathbb{R}}} % reals
\newcommand{\icomp}{\mathrm{i}}
\newcommand{\NN}{{\mathbb{N}}} % natural numbers {1, 2, ...}
\newcommand{\calP}{{\mathcal{P}}}

\newcommand{\bbR}{{\mathbb{R}}}
\newcommand{\bbT}{{\mathbb{T}}}
\newcommand{\bbN}{{\mathbb{N}}}
\newcommand{\bbD}{{\mathbb{D}}}
\newcommand{\bbC}{{\mathbb{C}}}
\newcommand{\bbA}{{\mathbb{A}}}
\newcommand{\FF}{{\mathbb{F}}} % field, finite field
\newcommand{\bbF}{{\mathbb{F}}}

\title{Infinite sequences with optimal diaphony, periodic $L_2$-discrepancy, and beyond}
\author{Peter Kritzer, Nicolas Nagel\footnote{Corresponding author: nicolas.nagel(AT)ricam.oeaw.ac.at}, Friedrich Pillichshammer}
\date{}

\begin{document}

\maketitle

\begin{abstract}
We investigate the periodic $L_2$-discrepancy of infinite sequences $\S_d$ in $[0,1)^d$ and its analytic counterpart, the diaphony. We prove that infinite order-2 digital sequences over $\mathbb{F}_2$ attain the optimal order $L_{2,N}^{{\rm per}}(\S_d) \le C_d (\log N)^{d/2}/N$ for all $N \in \mathbb{N}\setminus \{1\}$, matching known lower bounds for infinitely many $N \in \N$. This confirms the conjectured optimality of order-2 constructions. 
By this result, we improve upon previously known constructions using order-5 digital sequences, and reduce the underlying dimension for the interlacing construction from $5d$ to $2d$, significantly improving practicality. We establish our bounds within a broader framework of quasi-Monte Carlo  integration for periodic Besov spaces $S_{p,q}^rB(\mathbb{T}^d)$ with dominating mixed smoothness $r \in (1/p,2)$, where $p,q\in [1,\infty]$. Rules based on infinite order-2 digital sequences yield worst-case errors of order $(\log N)^{(d-1)(1-1/q)} / N^{ \min(r,1)}$ for $r \not=1$, and $(\log N)^{d(1-1/q)}/N$ for $r=1$, for all $N \in \N\setminus\{1\}$, while preserving extensibility in $N$.
\end{abstract}

\centerline{\begin{minipage}[hc]{130mm}{
{\em Keywords:} Discrepancy, diaphony, numerical integration, quasi-Monte Carlo methods, digital sequences, Besov spaces\\
{\em MSC 2020:} 11K38, 65C05, 65D30, 65Y20}
\end{minipage}}

\section{Introduction}

In quasi-Monte Carlo (QMC) methods, the accuracy of numerical integration hinges on how uniformly sample points cover the domain, in the most basic case the unit cube $[0,1)^d$. Evenly distributed point sets and sequences minimize clustering and gaps, producing more balanced coverage than pseudorandom points and thereby reducing the integration error. This is usually formalized by a Koksma-Hlawka-type inequality, which bounds the error by the product of a measure for the variation of the integrand (in the classical case, the variation in the sense of Hardy and Krause) and a suitable corresponding notion of discrepancy (in the classical case, the star-discrepancy), which measures the irregularity of distribution of a point set. Introductions to the QMC method are, e.g., provided by the books \cite{DT97,DKP22,DP10,LP14,N92} or the survey~\cite{DKS}. Nowadays, there are various notions of (geometric) discrepancies, which in many cases appear in ``integration-discrepancy dualities'' as the worst-case errors of QMC rules in corresponding function spaces (see \cite[Chapter~9]{NW10} or \cite[Section~4]{NP26}). In the present paper we consider the problem of numerical integration of periodic functions. Naturally, in this case, the periodic $L_2$-discrepancy is our preferred quality criterion for the underlying sample points. This kind of discrepancy has received increasing attention during the last decade, see, e.g., \cite{HKP20,HOe,HW12,KP22b,Lev,Na,NP26,P23}.

The periodic $L_2$-discrepancy is a quantitative measure for the irregularity of distribution of point sets and sequences modulo 1 which uses periodic rectangles as test sets. These are defined as follows. For $x,y\in [0,1]$ put
$$B(x,y)\coloneqq\begin{cases}
           [x,y) & \text{if $x\leq y$}, \\
           [0,y)\cup [x,1)& \text{if $x>y$,}
          \end{cases}$$
and for $\bsx=(x_1,\dots,x_d)$ and $\bsy=(y_1,\dots,y_d)$ in $[0,1]^d$ we set $B(\bsx,\bsy)\coloneqq B(x_1,y_1)\times \dots \times B(x_d,y_d)$. Let $\P_{d,N}=\{\bsx_0,\dots,\bsx_{N-1}\}$ be an $N$-point set in $[0,1)^d$ (which will be identified with the $d$-dimensional torus $\mathbb{T}^d$ in the following). Then the periodic $L_2$-discrepancy of $\P_{d,N}$ is defined as
$$ L_{2,N}^{\mathrm{per}}(\P_{d,N})\coloneqq \left(\int_{[0,1]^{2d}}\left|\frac{\#(\P_{d,N} \cap B(\bsx,\bsy))}{N}-\lambda(B(\bsx,\bsy))\right|^2\rd (\bsx,\bsy)\right)^{1/2},$$
where $\#(\cdot)$ denotes the number of points in a set and $\lambda$ the Lebesgue (uniform) measure of a set. For an infinite sequence $\S_d=(\bsx_n)_{n \ge 0}$ let $\S_{d,N}$ denote the set $\{\bsx_0,\ldots,\bsx_{N-1}\}$ consisting of the initial $N$ elements of the sequence and define $$L_{2,N}^{\mathrm{per}}(\S_d)\coloneqq  L_{2,N}^{\mathrm{per}}(\S_{d,N}).$$

The periodic $L_2$-discrepancy can be expressed in terms of exponential sums. For $\P_{d,N}$ as above we have
\begin{equation}\label{pr_dia}
L_{2,N}^{{\rm per}}(\P_{d,N})=\left(\frac{1}{3^d} \sum_{\bsh \in \ZZ^d\setminus\{\bszero\}} \frac{1}{r(\bsh)^2} \left|\frac{1}{N} \sum_{n=0}^{N-1} \exp(2 \pi \icomp \bsh \cdot \bsx_n)\right|^2\right)^{1/2},
\end{equation} 
where $\icomp=\sqrt{-1}$, where ``$\cdot$'' denotes the usual Euclidean inner product in $\RR^d$, and where for $\bsh=(h_1,\ldots,h_d)\in \ZZ^d$ we set 
\begin{equation*}
r(\bsh)\coloneqq \prod_{j=1}^d r(h_j) \ \ \ \mbox{ and } \ \ r(h_j)\coloneqq \left\{ 
\begin{array}{ll}
1 & \mbox{ if $h_j=0$},\\
\frac{2 \pi |h_j|}{\sqrt{6}} & \mbox{ if $h_j\not=0$.} 
\end{array}\right.
\end{equation*}
For a proof of this relation, see \cite[Theorem~1]{Lev}. %or \cite[p.~390]{HOe}.

Equation~\eqref{pr_dia} shows that the periodic $L_2$-discrepancy is---up to a multiplicative factor depending on $d$---asymptotically equivalent to the diaphony. Diaphony is another quantitative measure for the irregularity of distribution which has been introduced by Zinterhof~\cite{zint} in 1976 (see also \cite{DT97}). For a finite set $\P_{d,N}$ as above the diaphony is defined as $$F_N(\P_{d,N})\coloneqq \left(\sum_{\bsh \in \ZZ^d\setminus\{\bszero\}} \frac{1}{\rho(\bsh)^2} \left|\frac{1}{N} \sum_{n=0}^{N-1} \exp(2 \pi \icomp \bsh \cdot \bsx_n)\right|^2\right)^{1/2},$$ where $\rho(\bsh)\coloneqq \prod_{j=1}^d \max(1,|h_j|)$ for $\bsh=(h_1,\ldots,h_d) \in \ZZ^d$. For an infinite sequence $\S_d$ the diaphony $F_N(\S_d)$ is the diaphony of the initial $N$ elements of $\S_d$, i.e., $$F_N(\S_d)\coloneqq F_N(\S_{d,N})\quad \mbox{for all } N \in \NN.$$ 

We express the fact that periodic $L_2$-discrepancy and diaphony are asymptotically equivalent up to a factor depending only on $d$ in the form $$L_{2,N}^{{\rm per}}(\P_{d,N})  \asymp_d F_N(\P_{d,N}).$$ As a consequence, the periodic $L_2$-discrepancy can be understood as a geometrical interpretation of the diaphony, which is of analytic nature. From this point of view, the results on periodic $L_2$-discrepancy presented in this paper also directly apply to the diaphony and vice versa.

The periodic $L_2$-discrepancy/diaphony is also closely related to the worst-case error of QMC rules in suitable spaces of smooth periodic functions over $[0,1]^d$, see equation~\eqref{eq:wce_lper} in Section~\ref{sec:besov}. This is a major motivation for the study of the periodic $L_2$-discrepancy/diaphony and will be discussed in more detail in Section~\ref{sec:besov}.\\

We pause briefly to explain important notation that is used throughout this work. For functions $f,g:D \subseteq \mathbb{N} \rightarrow \mathbb{R}$ with $g \ge 0$ the notation $f(N) \lesssim g(N)$ means that there exists some $C>0$ such that $f(N) \le C g(N)$ for all $N \in D$. If we want to stress that $C$ depends on some parameters, say $a,b$, then this is indicated by writing $f(N) \lesssim_{a,b} g(N)$. If in addition also $f \ge 0$ and if we have $f(N) \lesssim g(N)$ and $g(N) \lesssim f(N)$, then we write $f(N) \asymp g(N)$ and similarly $f(N) \asymp_{a,b} g(N)$ if we want to emphasize the dependence on additional parameters $a, b$, etc.\\ 

The periodic $L_2$-discrepancy/diaphony of {\it finite} point sets has been studied in multiple papers, for example in \cite{HKP20,HOe,HMOU,KP22b}. In the present work we are interested in the periodic $L_2$-discrepancy/diaphony of \textit{infinite} sequences (in dimension $d$). It is known that for any $d \in \NN$ there exists a quantity $c_d>0$ such that for any sequence $\S_d$ in $[0,1)^d$ we have 
\begin{equation}\label{dia:lbd}
L_{2,N}^{{\rm per}}(\S_d) \ge c_d \, \frac{(\log N)^{d/2}}{N} \qquad \mbox{for infinitely many } N \in \NN\setminus\{1\}.
\end{equation}
%\peter{[Das $\setminus \{1\}$ erscheint mir fuer die untere Schranke wenig sinnvoll. Und die Schranke ist ja auch fuer $N=1$ trivial richtig.]}
This has been shown first in the context of diaphony by Pro{\u\i}nov~\cite{pro2000} (this work is only available in Bulgarian). A discussion of Pro{\u\i}nov's work can be found in \cite{kirk}. Later, a direct proof in terms of the periodic $L_2$-discrepancy has been provided in \cite{KP22a}.

On the other hand, for every $d \in \NN$ there exists an infinite sequence $\S_d$ in $[0,1)^d$ such that 
\begin{equation}\label{dia:ubd}
L_{2,N}^{{\rm per}}(\S_d) \lesssim_d \, \frac{(\log N)^{d/2}}{N} \qquad \mbox{for all } N \in \NN\setminus\{1\}.
\end{equation}
Combining the two equations \eqref{dia:lbd} and \eqref{dia:ubd} shows that 
$(\log N)^{d/2}/N$ is the exact asymptotic order of magnitude in $N$ of the optimal periodic $L_2$-discrepancy/diaphony for infinite sequences in dimension $d$. 

One-dimensional %(i.e., $d=1$) \peter{[Brauchen wir die Klammer wirklich? Mir kommt das irgendwie selbstverstaendlich vor...]}
infinite sequences whose periodic $L_2$-discrepancy satisfies the bound \eqref{dia:ubd} are given in, e.g., \cite{chafa,g96,pag,pro1988a,pg}. These constructions comprise the well-known van der Corput sequence as shown by Pro{\u\i}nov and Grozdanov~\cite{pg}. 

For arbitrary $d \in \NN$ the upper bound \eqref{dia:ubd} is obtained by the means of order-5 digital sequences over the finite field $\mathbb{F}_2$, as shown in \cite{P23}. The construction of an order-5 digital sequence is based on an ordinary digital sequence in the sense of Niederreiter~\cite{N87} in dimension $d'=5d$ via a procedure called interlacing (see Section~\ref{sec:dignet}). This is a major disadvantage of these sequences and makes the practical application of the existing result infeasible in higher dimensions. A remedy would be to reduce the order parameter, which can be seen as a complexity reduction in the construction of the sequences and, therefore, in the whole method. In fact, already in \cite{P23} it is conjectured that the optimal order \eqref{dia:ubd} of periodic $L_2$-discrepancy/diaphony of infinite sequences in dimension $d$ is achieved by means of order-2 (instead of order-5) digital sequences. In this way, the dimension of the underlying ordinary digital sequences is reduced from $d'=5 d$ to $d'=2 d$, which represents a significant improvement in terms of practicality.

In this paper, we shall establish our result in a much broader framework and study QMC integration of multivariate periodic functions by means of infinite sequences. A QMC rule approximates the integral of a function $f: \mathbb T^d \rightarrow \mathbb R$ using only a finite set of sample nodes $\calP_{d, N}$ in $\mathbb T^d$, $\#\calP_{d, N} = N$, by means of
$$
\int_{\mathbb T^d} f(\bsx) \, \mathrm d\bsx \approx \frac1N \sum_{\bsx \in \calP_{d, N}} f(\bsx).
$$
The quality of this approximation for functions $f$ in a Banach space $F \subseteq \{f: \mathbb T^d \rightarrow \bbR\}$ of test functions with norm $\|\cdot\|_F$ depends on the set of sample nodes and can be quantified via the worst-case error
$$
\wce(\calP_{d, N}, F) \coloneqq \sup_{\substack{f \in F \\ \|f\|_F \leq 1}} \left|\int_{\mathbb T^d} f(\bsx) \, \mathrm d\bsx - \frac1N \sum_{\bsx \in \calP_{d, N}} f(\bsx)\right|.
$$
The underlying function space setting for our analysis is based on Besov spaces with dominating mixed smoothness, $S_{p,q}^rB(\mathbb{T}^d)$, with smoothness parameter $r \in (1/p,2)$. The necessary definitions will be given in Section~\ref{sec:besov}. In our main result, which is Theorem~\ref{thm1} in Section~\ref{sec:main}, we show that QMC rules based on infinite order-2 digital sequences achieve a convergence rate of the worst-case error of order $N^{-\min(r,1)} (\log N)^{(d-1)(1-1/q)}$ if $r \not=1$ and $N^{-1} (\log N)^{d(1-1/q)}$ if $r=1$ (this case comprises the periodic $L_2$-discrepancy/diaphony as a special case;  see Corollary~\ref{cor:perL2}, also in Section~\ref{sec:main}) for all $N \in \NN\setminus\{1\}$. The fundamentals of digital sequences necessary for this article are explained in Section~\ref{sec:dignet}. Our result can be seen as the ``infinite sequence variant'' of the result for finite point sets in \cite{HMOU}. Rules based on infinite sequences have the advantage that they are extensible in the number of integration nodes $N$. This means that in order to increase $N$ in a QMC rule, one only needs to evaluate the integrand at the additional nodes, and there is no need to discard previous function evaluations. QMC rules that are extensible in $N$ in this sense are sometimes referred to as rules of ``open type''; see, e.g., \cite[Section~2.3]{DKS}.

\section{Besov spaces and the Faber basis}\label{sec:besov}

In this section we collect basic definitions and results from \cite{HMOU} regarding Besov spaces and the Faber basis. The definitions require extensive explanations; we try to minimize the necessary effort while ensuring proper understanding.

For the analysis we will use function spaces based on integrability and regularity assumptions on the set $F$ of test functions. Concretely, for $p,q\in [1,\infty]$ consider the Besov space $S^r_{p, q}B(\mathbb T^d)$ of dominating mixed smoothness $r$. These spaces and their respective norms $\|\cdot\|_{S_{p, q}^rB(\bbT^d)}$ can be defined via dyadic decompositions of the Fourier spectrum or via decay conditions on moduli of continuity, see \cite{HMOU} for details. Relevantly for us, in the range $r \in (1/p,2)$ functions from these spaces fulfill a decay condition when represented in a certain basis. We present this characterization in greater detail. In what follows, let $\bbN_0 \coloneqq \{0\} \cup \bbN$ be the nonnegative integers and $\bbN_{-1} \coloneqq \{-1\} \cup \bbN_0$.

Consider the hat function
$$
v(x) \coloneqq \begin{cases}
    1-|2x-1| & \mbox{if } x \in [0,1], \\
    0 & \text{else,}
\end{cases}
$$
for $x \in \bbR$. The Faber basis (also known as hierarchical basis) is given by
$$
v_{-1, 0}(x) \coloneqq  1 \qquad \mbox{ and }\qquad v_{j, k}(x) \coloneqq  v(2^jx-k)
$$
for $j \in \bbN_0$, $k \in \bbD_j \coloneqq \{0, 1, \dots, 2^j-1\}$, where for notational convenience we set $\bbD_{-1} \coloneqq \{0\}$. 

We also require second-order differences. For univariate functions $f:\mathbb{T} \rightarrow \mathbb{C}$ for $x \in \bbT$ and $h \in [0,1]$, put 
\begin{equation}\label{2dif:op:1}
\Delta_h(f,x)\coloneqq f(x)-2 f(x+h)+f(x+2 h).    
\end{equation}
It was shown in \cite{Fab} that every continuous function $f: \mathbb T \rightarrow \bbR$ (recalling that we identify $\bbT$ with $[0, 1)$) can be represented pointwise as
$$
f(x) = \sum_{j \in \bbN_{-1}} \sum_{k \in \bbD_j} d_{j, k}(f) v_{j, k}(x)
$$
with coefficients $d_{j, k}(f)$ given by $d_{-1, 0}(f) \coloneqq  f(0)$, and for $j \in \bbN_0, k \in \bbD_j$ via second differences
$$
d_{j, k}(f) \coloneqq  -\frac12 \Delta_{2^{-j-1}}\left(f,\frac{k}{2^j}\right).
$$

In the case of $\bbT^d$ for general $d \in \bbN$ one may proceed via tensorization. Define for $\bsj =(j_1,\ldots,j_d) \in \bbN_{-1}^d$ and
$\bsk \in \bbD_{\bsj} \coloneqq \prod_{i=1}^d \bbD_{j_i}$ the tensor Faber basis
$$
v_{\bsj, \bsk}(\bsx) \coloneqq \prod_{i=1}^d v_{j_i, k_i}(x_i).
$$
Furthermore, define for $f:\bbT^d \rightarrow \bbC$, $\vv \subseteq [d]$, and $\bsh=(h_1,\ldots,h_d) \in [0,1]^d$ the $\vv$-mixed difference operator $\Delta_{\bsh}^{\vv}$ as $$\Delta_{\bsh}^{\vv}\coloneqq \prod_{i \in \vv} \Delta_{h_i,i} \qquad \mbox{ and } \qquad \Delta_{\bsh}^{\emptyset}\coloneqq {\rm id},$$ where ${\rm id}(f)\coloneqq f$ and $\Delta_{h_i,i}$ is the univariate second-order difference operator \eqref{2dif:op:1}, applied to the $i$-th coordinate of $f$ with the other variables unchanged. Then, for continuous functions $f:\bbT^d \rightarrow \bbC$ one obtains
$$
f(\bsx) = \sum_{\bsj \in \bbN_{-1}^d} \sum_{\bsk \in \bbD_\bsj} d_{\bsj, \bsk}(f) v_{\bsj, \bsk}(\bsx)
$$
pointwise, where (see \cite[Equation~(3.6)]{HMOU}), for $\bsj \in \bbN_{-1}^d$ and $\bsk \in \bbD_\bsj$, $$d_{\bsj, \bsk}(f) = \left(-\frac{1}{2}\right)^{|\vv(\bsj)|} \Delta^{\vv(\bsj)}_{(2^{-j_1-1},\ldots,2^{-j_d-1})}\left(f,\left(\frac{k_1}{2^{(j_1)_+}},\ldots,\frac{k_d}{2^{(j_d)_+}} \right)\right),$$ with $\vv(\bsj)\coloneqq\{i : j_i \not=-1\}$ and, for $a \in \bbR$,
\begin{align} \label{pos_part}
    (a)_+\coloneqq\max(a,0).
\end{align}

The main tool in \cite{HMOU} is the following result, characterizing the Besov spaces $S_{p, q}^rB(\mathbb T^d)$ of regularity $r<2$ in terms of the decay of the coefficients $d_{\bsj, \bsk}(f)$. For its statement in Proposition~\ref{prop:Besov_norm}, define the auxiliary norm
$$
\|f\|_{s_{p, q}^r} \coloneqq \left[\sum_{\bsj \in \bbN_{-1}^d} 2^{|\bsj|_1 (r-1/p) q} \left(\sum_{\bsk \in \bbD_\bsj} |d_{\bsj, \bsk}(f)|^p\right)^{q/p}\right]^{1/q}
$$
with the usual modifications if $p = \infty$ or $q = \infty$. Here, we write $|\bsj|_1 \coloneqq |j_1| + \dots + |j_d|$.

\begin{prop}\label{prop:Besov_norm}
Let $p,q \in [1,\infty]$ and let $r \in (1/p,2)$. There is a constant $c_{p, q, r, d} > 0$ depending only on $p, q, r$, and $d$, such that
$$
\|f\|_{S_{p, q}^rB(\bbT^d)} \geq c_{p, q, r, d} \|f\|_{s_{p, q}^r}
$$
for all $f \in S_{p, q}^rB(\bbT^d)$. If in addition $r < 1+1/p$, there is a constant $C_{p, q, r, d} > 0$ depending only on $p, q, r$, and $d$, such that
$$
\|f\|_{S_{p, q}^rB(\bbT^d)} \leq C_{p, q, r, d} \|f\|_{s_{p, q}^r}
$$
for all $f \in S_{p, q}^rB(\bbT^d)$.
\end{prop}

A proof can be found in \cite[Proposition~3.4 and 3.5]{HMOU}.

\begin{rem}\rm
The condition $1/p < r$ is required to have the continuous embedding $S_{p, q}^rB(\mathbb T^d) \hookrightarrow C(\mathbb T^d)$ into the space of continuous functions. The upper bound $r < 2$ comes from the intrinsic regularity of the Faber functions $v_{\bsj, \bsk}$ with kinks at dyadic rationals. A similar decay behavior for the endpoint case $r=2$ in the Sobolev setting was recently given in \cite{KNU} (based on \cite{BG}) and characterizations via other basis functions were considered in \cite{SU}.
\end{rem}

In the case of $p=q=2$ one has $S_{2, 2}^rB(\mathbb T^d) = H_{\text{mix}}^r(\mathbb T^d)$ with the usual Sobolev space of integrability $2$ and dominating mixed smoothness $r > 1/2$ (also called Korobov space in the literature). For integral smoothness parameters $r \in \bbN$ these Sobolev spaces are given by functions with weak partial derivatives of order up to $r$ in $L_2(\mathbb T^d)$, see \cite{DTU}. For general $r > 1/2$ they can be characterized via decay conditions on Fourier coefficients, see e.g. \cite{D07}. The space $H_{\text{mix}}^r(\mathbb T^d)$ has the structure of a reproducing kernel Hilbert space \cite{BT}, which significantly facilitates the analysis of the QMC rule (see e.g. \cite{D07,DP10}). It is well known that
\begin{align} \label{eq:wce_lper}
    \wce(\calP_{d, N}, H_{\text{mix}}^1(\mathbb T^d)) \asymp_d L_{2,N}^{{\rm per}}(\P_{d,N})
\end{align}
with implied factors only depending on $d$, see, e.g., \cite{HKP20, HOe}. We thus see how bounds on the worst-case error $\wce(\calP_{d, N}, S_{p, q}^rB(\mathbb T^d))$ can be used to study the periodic $L_2$-discrepancy of point sets or sequences.

The worst-case error $\wce(\calP_{d, N}, S_{p, q}^rB(\mathbb T^d))$ may now be bounded in terms of the specific errors of the tensor Faber functions
\begin{equation}\label{def:cjk}
c_{\bsj, \bsk}(\calP_{d, N}) \coloneqq \frac1N \sum_{\bsx \in \calP_{d, N}} v_{\bsj, \bsk}(\bsx) - \int_{\bbT^d} v_{\bsj, \bsk}(\bsx) \, \mathrm d\bsx.
\end{equation}
The proof of Lemma~\ref{est:err:net} below is a simple application of Proposition~\ref{prop:Besov_norm} and H\"older's inequality; see \cite[Proof of Theorem~5.3]{HMOU} for details, and in particular \cite[Equation~(5.8)]{HMOU}.

\begin{lem}\label{est:err:net}
Let $p, q \in [1,\infty]$ and $r \in (1/p ,2)$. Also let $p', q' \in [1,\infty]$ be the H\"older conjugates of $p$ and $q$, respectively given by $1/p+1/p'=1$ and $1/q+1/q'=1$. Then for all $N$-point sets $\calP_{d, N}$ in $\bbT^d$ it holds that
    \begin{align*} 
        \wce(\calP_{d, N}, S_{p, q}^rB(\mathbb T^d)) \lesssim_{p,q,r,d} \left[\sum_{\bsj \in \bbN_{-1}^d} 2^{-|\bsj|_1 (r-1/p) q'} \left(\sum_{\bsk \in \bbD_\bsj} |c_{\bsj, \bsk}(\calP_{d, N})|^{p'}\right)^{q'/p'}\right]^{1/q'}
    \end{align*}
    with an implied factor only depending on $p, q, r$, and $d$, and the usual adaptions if $p=1$ or $q=1$.
\end{lem}

It remains to discuss the relation of the errors $c_{\bsj, \bsk}(\calP_{d, N})$ of the Faber basis functions to the Haar coefficients of the local discrepancy function, which are an important analytic tool in the study of digital nets and sequences. The local discrepancy function $D_{\calP_{d, N}}: [0, 1)^d \rightarrow \bbR$ of an $N$-point set $\calP_{d, N}$ in $[0, 1)^d$ is given by
$$
D_{\calP_{d, N}}(\bsx) \coloneqq \frac{\#(\calP_{d, N} \cap [\mathbf 0, \bsx))}{N}  - \prod_{i=1}^d x_i \qquad \mbox{for } \bsx \in [0,1]^d,
$$ 
where $[\mathbf 0, \bsx) \coloneqq [0, x_1)\times \cdots \times [0,x_d)$.

Let
\begin{align} \label{eq:Ijk}
    I_{j, k} \coloneqq \left[\frac{k}{2^j}, \frac{k+1}{2^j}\right)
\end{align}
be the dyadic interval for $j \in \bbN_0, k \in \bbD_j$. Observe that $I_{j, k} = I_{k+1, 2k} \cup I_{k+1, 2k+1}$ disjointly partitions the interval $I_{j, k}$ into a left half $I_{j+1, 2k}$ and a right half $I_{k+1, 2k+1}$. For $j \in \bbN_{-1}$ and $k \in \bbD_j$ define the {Haar basis} $h_{j, k}: [0, 1) \rightarrow \bbR$ by $h_{-1, 0}(x) \coloneqq 1$ and, if $j \not= -1$, by
$$
h_{j, k}(x) \coloneqq \left\{\begin{array}{rl}
    1 & \mbox{if } x \in I_{j+1, 2k}, \\
    -1 & \mbox{if } x \in I_{j+1, 2k+1}, \\
    0 & \mbox{else}.
\end{array}
\right.
$$
Also, define its tensor product analogue
$$
h_{\bsj, \bsk}(\bsx) \coloneqq \prod_{i=1}^d h_{j_i, k_i}(x_i)\qquad \mbox{for } \bsj \in \bbN_{-1}^d,\, \bsk \in \bbD_\bsj,\, \bsx \in [0, 1)^d.
$$
%Since the Haar system forms an orthogonal basis, the discrepancy function $D_{\calP_{d, N}}$ may now be expressed by the in $L_2([0, 1)^d)$ converging series \todo{Why $\sim$?}
%$$
%D_{\calP_{d, N}} \fritz{\sim} \sum_{\bsj \in \bbN_{-1}^d} 2^{|\bsj|_+} \sum_{\bsk \in \bbD_\bsj} \mu_{\bsj, \bsk}(D_{\calP_{d, N}}) h_{\bsj, \bsk}
%$$
The Haar coefficients of the local discrepancy function are defined as 
$$
\mu_{\bsj, \bsk}(D_{\calP_{d, N}}) \coloneqq \int_{[0, 1)^d} D_{\calP_{d, N}}(\bsx) h_{\bsj, \bsk}(\bsx) \, \mathrm d\bsx.
$$

The next lemma, which is \cite[Lemma~5.2]{HMOU}, shows that the Haar coefficients $\mu_{\bsj, \bsk}(D_{\calP_{d, N}})$ can be used to determine the errors $c_{\bsj, \bsk}(\calP_{d, N})$ of the Faber basis functions from \eqref{def:cjk}. Here we write, like in \eqref{pos_part}, $|\bsj|_+\coloneqq(j_1)_++\cdots+(j_d)_+$.

\begin{lem}\label{le1}
For an $N$-point set $\calP_{d, N}$ in $[0,1)^d$ we have:
\begin{enumerate}
    \item[(i)] If $\bsj \in \bbN_0^d$ and $\bsk \in \bbD_{\bsj}$, then $$\mu_{\bsj,\bsk}(D_{\calP_{d, N}})=(-1)^d 2^{-|\bsj|_1}  c_{\bsj, \bsk}(\calP_{d, N}).$$ 
    \item[(ii)] If $\bsj \in \bbN_{-1}^d \setminus \bbN_0^d$ and $\bsk \in \bbD_{\bsj}$, then $$\mu_{\bsj,\bsk}(D_{\overline{\calP}_{d, N}})=(-1)^s 2^{-|\bsj|_+} c_{\bsj, \bsk}(\P),$$ where $\overline{\calP}_{d, N}$ denotes the set of points obtained from the projection of the elements of $\P_{d, N}$ onto those $s$ coordinates $x_i$ where $j_i\not=-1$. Furthermore, $\mu_{\bsj,\bsk}(D_{\overline{\P}_{d, N}})$ is the Haar coefficient with respect to the $s$-variate function $D_{\overline{\P}_{d, N}}$. 
\end{enumerate}
\end{lem}

\section{Digital nets and sequences}\label{sec:dignet}

The concepts of digital nets and sequences over finite fields were introduced by Nieder\-rei\-ter~\cite{N87} in 1987. Detailed introductions to this topic can be found in the books \cite{DP10,LP14,N92}. In the present work, we restrict ourselves to the binary case. Let $\mathbb{F}_2$ be the finite field of order~2, identified with the set $\{0,1\}$ equipped with arithmetic operations modulo~2. In the following, let $[d]\coloneqq\{1,\ldots,d\}$, where $d \in \mathbb{N}$.

\paragraph{The digital construction scheme.}

We begin with the definition of digital nets according to Niederreiter, which we present in a slightly more general form here. For $n,q,d \in \N$ with $q \ge n$ let $C_1,\ldots, C_d \in \mathbb{F}_2^{q \times n}$ be $q \times n$ matrices over $\mathbb{F}_2$. For $k \in \{0,\ldots ,2^n-1\}$ with binary expansion $k = k_0 + k_1 2 + \cdots + k_{n-1} 2^{n-1}$, where $k_j \in \{0,1\}$, we define the binary digit vector $\vec{k} = (k_0, k_1, \ldots, k_{n-1})^\top \in \mathbb{F}_2^n$ (the symbol $\top$ means the transpose of a vector or a matrix; hence $\vec{k}$ is a column-vector). Then compute
\begin{equation}\label{matrix_vec_net}
C_j \vec{k} =:(x_{j,k,1}, x_{j,k,2},\ldots,x_{j,k,q})^\top \quad \mbox{for } j \in [d],
\end{equation}
where the matrix vector product is evaluated over $\mathbb{F}_2$. We interpret the entries $x_{j,k,i}$ of the resulting vector in $\mathbb{F}_2^q$ in \eqref{matrix_vec_net} as binary digits in $\{0,1\}$, and put
\begin{equation*}
x_{j,k} \coloneqq \frac{x_{j,k,1}}{2} + \frac{x_{j,k,2}}{2^2} + \cdots + \frac{x_{j,k,q}}{2^q} \in [0,1)
\end{equation*}
and $\boldsymbol{x}_k \coloneqq (x_{1,k}, \ldots, x_{d,k})$, which is a point in $[0,1)^d$. The point set $\P_{d, 2^n}=\{\boldsymbol{x}_0,\ldots,\boldsymbol{x}_{2^n-1}\}$ constructed this way is called a {digital net (over $\mathbb{F}_2$) with generating matrices} $C_1,\ldots,C_d$. A digital net $\P_{d, 2^n}$ consists of $2^n$ points in $[0,1)^d$.\\

A variant of digital nets are so-called digitally shifted digital nets (see \cite[Section~4.4.4]{DP10}). Here, one additionally selects a so-called digital shift vector for each component $j \in [d]$ of the form $\vec{\sigma}_j=(\sigma_{j,1},\sigma_{j,2},\ldots)^{\top} \in \mathbb{F}_2^{\mathbb{N}}$ (for our purposes it suffices to restrict ourselves to the case where only a finite number of components differ from zero) and replaces \eqref{matrix_vec_net} by
\begin{equation*}
{C_j \vec{k} \choose \vec{0}}+\vec{\sigma}_j =:(x_{j,k,1}, x_{j,k,2},x_{j,k,3},\ldots,)^\top \in \mathbb{F}_2^{\mathbb{N}} \quad \mbox{for } j \in [d],
\end{equation*}
where ${C_j \vec{k} \choose \vec{0}}$ in $\mathbb{F}_2^{\mathbb{N}}$ is the $\mathbb{F}_2^q$-vector $C_j \vec{k}$ concatenated with consecutive zeros, and puts
\begin{equation*}
x_{j,k} = \frac{x_{j,k,1}}{2} + \frac{x_{j,k,2}}{2^2} + \frac{x_{j,k,3}}{2^3} +\cdots  \in [0,1).
\end{equation*}

The main objects of our study are digital sequences, which are infinite versions of digital nets. Let now $C_1,\ldots, C_d \in \mathbb{F}_2^{\mathbb{N} \times \mathbb{N}}$ be $\mathbb{N} \times \mathbb{N}$ matrices over $\mathbb{F}_2$. For $C_j = (c_{j,k,\ell})_{k, \ell \in \mathbb{N}}$ it is assumed that for each $\ell \in \mathbb{N}$ there exists a $K(\ell) \in \mathbb{N}$ such that $c_{j,k,\ell} = 0$ for all $k > K(\ell)$. For $k \in \N_0$ with binary expansion $k = k_0 + k_1 2 + \cdots + k_{m-1} 2^{m-1} \in \mathbb{N}_0$,  define the infinite binary digit vector $\vec{k} = (k_0, k_1, \ldots, k_{m-1}, 0, 0, \ldots )^\top \in \mathbb{F}_2^{\mathbb{N}}$. Then compute, over $\mathbb{F}_2$,
\begin{equation*}
C_j \vec{k}=:(x_{j,k,1}, x_{j,k,2},\ldots)^\top \quad \mbox{for } j \in [d]
\end{equation*}
and put, similarly to before,
\begin{equation*}
x_{j,k} = \frac{x_{j,k,1}}{2} + \frac{x_{j,k,2}}{2^2} + \cdots ,
\end{equation*}
and $\boldsymbol{x}_k = (x_{1,k}, \ldots, x_{d,k}) \in [0,1)^d$. A sequence $\S_d=(\boldsymbol{x}_k)_{k \ge 0}$ constructed this way is called a {digital sequence (over $\mathbb{F}_2$) with generating matrices} $C_1,\ldots,C_d$. Since $c_{j,k,\ell}=0$ for all $k$ large enough, the numbers $x_{j,k}$ are always dyadic rationals, i.e., have a finite dyadic expansion, which belong to $[0,1)$. 

\paragraph{Higher order nets and sequences.}%\label{sec_honetssequ}

The choice of the respective generating matrices is crucial for the distribution quality of digital nets and sequences. The following definitions put some restrictions on $C_1,\ldots ,C_d$ with the aim to quantify the quality of equidistribution of the digital net or sequence.

\begin{defin}\rm\label{def_net}
Let $n, q, \alpha \in \N$ with $q \ge \alpha n$ and let $t \in \N_0$ such that $0 \le t \le \alpha n$. Let $C_1,\ldots, C_d \in \mathbb{F}_2^{q \times n}$. Denote the $i$-th row vector of the matrix $C_j$ by $\vec{c}_{j,i}$. The $\vec{c}_{j,i}$ are row vectors of length $n$ with components in $\mathbb{F}_2$. If for all integers $1 \le i_{j,\nu_j} < \cdots <
i_{j,1} \le q$ with 
\[
\sum_{j = 1}^d \sum_{l=1}^{\min(\nu_j,\alpha)} i_{j,l}  \le
\alpha n - t,
\] 
the vectors
\[
\vec{c}_{1,i_{1,\nu_1}}, \ldots, \vec{c}_{1,i_{1,1}}, \ldots,
\vec{c}_{d,i_{d,\nu_d}}, \ldots, \vec{c}_{d,i_{d,1}}
\] 
are linearly independent
over $\mathbb{F}_2$, then the digital net with generating matrices
$C_1,\ldots, C_d$ is called an {order-$\alpha$ digital $(t,n,d)$-net over $\mathbb{F}_2$}. The parameter $t$ is referred to as the quality parameter.
\end{defin}

Furthermore, we consider digital sequences whose initial segments are order-$\alpha$ digital $(t,n,d)$-nets over $\mathbb{F}_2$.

\begin{defin}\rm%\label{def_seq}
Let $\alpha \in \N$ and $t \in \N_0$. Let $C_1,\ldots, C_d \in \mathbb{F}_2^{\mathbb{N} \times \mathbb{N}}$ and let $C_{j, \alpha n \times n}$ denote the left upper $\alpha n \times n$ submatrix of  $C_j$ for $j\in [d]$. If for all $n > t/\alpha$ the matrices $C_{1, \alpha n \times n},\ldots, C_{d, \alpha n \times n}$ generate an order-$\alpha$ digital $(t,n,d)$-net over $\mathbb{F}_2$, then the digital sequence with generating matrices $C_1,\ldots, C_d$ is called an {order-$\alpha$ digital $(t,d)$-sequence over $\mathbb{F}_2$}.
\end{defin}

Order-1 (i.e., $\alpha=1$) in the above definitions refers to the classical definition of digital $(t,n,d)$-nets and $(t,d)$-sequences by Niederreiter~\cite{N87} (see also \cite{N92} and the introductory books \cite{DP10,LP14}). The extension to higher order $\alpha \ge 1$ was introduced by Dick~\cite{D07,D08}. For a geometrical interpretation of higher order digital nets and sequences we refer to \cite{N92,DP10,LP14} for the classical case $\alpha=1$ and \cite{DB09} for general $\alpha$. Roughly speaking, the definitions imply that special intervals or unions of intervals of prescribed volume contain the correct share of points with respect to a perfect uniform distribution.

Definition~\ref{def_net} shows that if $\P_{d, 2^n}$ is an order-$\alpha$ digital $(t,n,d)$-net, then for any $t \le t' \le \alpha n$, $\P_{d, 2^n}$ is also an order-$\alpha$ digital $(t',n,d)$-net. An analogous result also applies to higher order digital sequences.

We point out that a digital net can be an order-$\alpha$ digital $(t,n,d)$-net over $\mathbb{F}_2$ and at the same time an order-$\alpha'$ digital $(t',n,d)$-net over $\mathbb{F}_2$ for $\alpha'\not = \alpha$. The quality parameter $t$ may depend on $\alpha$ (i.e., $t=t(\alpha)$). The same is true for digital sequences. In particular, \cite[Theorem~4.10]{D08} implies that an order-$\alpha$ digital $(t,n,d)$-net is an order-$\alpha'$ digital $(t',n,d)$-net for all $1 \le \alpha' \le \alpha$ with 
\begin{equation*}%\label{eq_t_tprime}
t' = \lceil t \alpha'/\alpha \rceil \le t.
\end{equation*}
The same result applies to order $\alpha$ digital $(t,d)$-sequences which are also order $\alpha'$ digital $(t',d)$-sequences for $1 \le \alpha' \le \alpha$ and $t'$ as above. In other words, $t(\alpha') = \lceil t(\alpha) \alpha'/\alpha \rceil$ for all $1 \le \alpha' \le \alpha$. For more details, see \cite{D07,D08} or \cite[Chapter~15]{DP10}.

\paragraph{Explicit constructions of order-2 digital sequences.}%\label{sec_exp_constr}

Explicit constructions of order-$\alpha$ digital nets and sequences have been provided by Dick~\cite{D07,D08}. For our purposes, it suffices to restrict ourselves to the case $\alpha=2$. 

In order to construct a sequence in dimension $d$ we start with $2 d$ (in general, $\alpha d$) generating matrices $C_1, \ldots, C_{2 d}$ of a digital sequence over $\mathbb{F}_2$. Let $\vec{c}_{j,k}$ denote the $k$-th row of $C_j$. We now define matrices $E_1,\ldots, E_d$, where the $k$-th row of $E_j$ is denoted by $\vec{e}_{j,k}$, in the following way. For all $j \in [d]$, $u \in \N_0$, and $v \in \{1,2\}$, set
\begin{equation*}
\vec{e}_{j,2 u + v} = \vec{c}_{2 (j-1) + v, u+1}.
\end{equation*}
We illustrate the construction for $d=1$. In this case we have
\[C_1=\left(\begin{array}{c} 
             \vec{c}_{1,1}\\
             \vec{c}_{1,2}\\
             \vdots 
            \end{array}\right), \ C_2=\left(\begin{array}{c} 
             \vec{c}_{2,1}\\
             \vec{c}_{2,2}\\
             \vdots 
            \end{array}\right) \ \Rightarrow \ 
           E_1=\left(\begin{array}{c} 
             \vec{c}_{1,1}\\
             \vec{c}_{2,1}\\
             \vec{c}_{1,2}\\
             \vec{c}_{2,2}\\
             \vdots 
            \end{array}\right).\]
This procedure is called {interlacing} (with interlacing factor $\alpha=2$).

\begin{rem}\rm
Recall that for generating matrices of digital sequences we demand $c_{j,k,\ell}=0$ for all $k > K(\ell)$. Let $E_j = (e_{j,k,\ell})_{k, \ell \in \mathbb{N}}$. Then the construction yields $e_{j,k,\ell} = 0$ for all $k > 2 K(\ell)$. 
\end{rem}

From \cite[Theorems~4.11 and~4.12]{D07} we obtain the following result.

\begin{prop}
If $C_1,\ldots,C_{2d} \in\mathbb{F}_2^{\N \times \N}$ generate an order-1 digital $(t',2d)$-sequence over $\mathbb{F}_2$, then $E_1,\ldots,E_d \in\mathbb{F}_2^{\N \times \N}$ generate an order-2 digital $(t,d)$-sequence over $\mathbb{F}_2$ with $t = 2 t' + d$.
\end{prop}

Examples of explicit constructions of suitable generating matrices of order-1 sequences over $\mathbb{F}_2$ were obtained by %\todo{auch Faure? P: Waere dafuer F: Nein, Faure braucht h\"ohere Basis, geht nicht mit $\mathbb{F}_2$ }
Sobol'~\cite{S67}, Niederreiter~\cite{N87,N92}, Niederreiter and Xing~\cite{NX96}, and others. An overview is presented in \cite[Chapter~8]{DP10}. Any of these constructions is sufficient for our purpose; however, for completeness, we briefly describe a special case of a construction of generating matrices by Tezuka~\cite{T93}, which is a generalization of Sobol's~\cite{S67} and Niederreiter's~\cite{N87} constructions, respectively. 

We explain how to construct the entries $c_{j,k,\ell} \in \mathbb{F}_2$ of the generating matrices $C_j = (c_{j,k,\ell})_{k,\ell \ge 1}$ for $j \in [d']$ (for our purpose, $d'=2d$). To this end, for $j \in [d']$, choose a polynomial $p_j \in \mathbb{F}_2[x]$ as the $j$-th irreducible polynomial in a list of irreducible polynomials over $\mathbb{F}_2$. The list is assumed to be  sorted in increasing order according to the degrees $e_\ell = \deg(p_\ell)$ of the polynomials $p_\ell$, i.e., $e_1 \le e_2 \le \cdots \le e_{d'}$ (the ordering of polynomials with the same degree is irrelevant). 

Let $j \in [d']$ and $k \in \mathbb{N}$. Take $i-1$ and $z$ to be the main term and the remainder when we divide $k-1$ by $e_j$, respectively, so that $k-1  = (i-1) e_j + z$, with $z \in \{0,\ldots,e_j-1\}$. Now consider the Laurent series expansion %\nicolas{[Seems like a better notation would be $\bbF_2[[x^{-1}]]$, but maybe it would be best to just avoid this notation entirely here.]}
\begin{equation*}
\frac{x^{e_j-z-1}}{p_j(x)^i} = \sum_{\ell =1}^\infty a_\ell(i,j,z) x^{-\ell}. %\in \mathbb{F}_2((x^{-1})).
\end{equation*}
For $\ell \in \mathbb{N}$ we set 
\begin{equation*}%\label{def_sob_mat}
c_{j,k,\ell} = a_\ell(i,j,z).
\end{equation*}
Every digital sequence with generating matrices $C_j = (c_{j,k,\ell})_{k,\ell \ge 1}$ for $j \in [d']$ obtained this way is a special instance of a Sobol' sequence, which in turn is a special instance of a so-called generalized Niederreiter sequence (see \cite[Equation~(3)]{T93}). Note that in the construction above we always have $c_{j,k,\ell}=0$ for $k > \ell$. The quality parameter of these sequences is known to be $t' = \sum_{j=1}^{d'} (e_j-1)$, see \cite[Chapter~4.5]{N92} for the case of Niederreiter sequences.

\begin{rem}\rm
Let $C_1,\ldots,C_{2d}$ be $\mathbb{N}\times \mathbb{N}$-matrices constructed according to Tezuka's method as described above. Let $E_1,\ldots,E_d$ be obtained from an order-2 interlacing of $C_1,\ldots,C_{2 d}$. Then we have $e_{j,k,\ell}=0$ for all $k > 2 \ell$, where $e_{j,k,\ell}$ is the entry in row $k$ and column $\ell$ of the matrix $E_j$.  The matrices $E_1,\ldots,E_d$ generate an order-2 digital $(t,d)$-sequence over $\mathbb{F}_2$ with $t=2 \sum_{j=1}^{2 d} (e_j-1) + d$. 
\end{rem}

\paragraph{Haar coefficients of the local discrepancy of order-2 nets and sequences.} Here we collect bounds on the Haar coefficients of the local discrepancy function of (digitally shifted) order-2 nets and sequences. By Lemma~\ref{le1} these bounds may be used to obtain bounds for the corresponding errors of the tensor Faber functions \eqref{def:cjk}.

\begin{lem} \label{haarcoeffdignets}
Let $\P_{d,2^n}$ be a digitally shifted order-2 digital $(t,n,d)$-net over $\mathbb{F}_2$. Let $\bsj\in\N_{-1}^d$ with $|\bsj|_1+t/2\leq n$ and $\bsk\in\D_{\bsj}$. Then
\[ |\mu_{\bsj,\bsk}(D_{\P_{d,2^n}})| \lesssim 2^{-2n+t}(2n-t-2|\bsj|_1)^{d-1}. \]
\end{lem}

\begin{proof}
The lemma is a slight generalization of \cite[Lemma~5.9]{M15}. The result was originally proved for order-$2$ digital $(t,n,d)$-nets. The extension to digitally shifted order $2$ digital $(t,n,d)$-nets follows by almost exactly the same arguments as the proof of \cite[Lemma~5.9]{M15}. The statement in the present form can also be found in \cite[Lemma~3.3]{DHMP16}. 
\end{proof}

The next lemma gives a bound for order-2 digital sequences. The dyadic boxes $I_{\bsj, \bsk} \coloneqq \prod_{i=1}^d I_{j_i, k_i}$ appearing below are given with $I_{j, k}$ as in \eqref{eq:Ijk}. In the following, let $\ld N \coloneqq \log_2 N$ denote the logarithm to the base $2$.

\begin{lem}\label{le131416}
Let $\S_d$ be an order-2 digital $(t,d)$-sequence over $\mathbb{F}_2$ with generating matrices $E_i=(e_{i,k,\ell})_{k,\ell \ge 1}$ for which $e_{i,k,\ell}=0$ for all $k > 2\ell$ and for all $i\in [d]$, and let $\S_{d, N}$ be the point set consisting of the first $N$ elements of $\S_d$. 
\begin{enumerate}
\item Let $|\bsj|_1+t/2\geq \lfloor \ld N \rfloor$ and set $\bbA_\bsj \coloneqq \{\bsk \in \bbD_\bsj: I_{\bsj, \bsk} \cap \S_{d, N} \neq \emptyset\}$. %if $I_{\bsj,\bsk}$ contains points of $\S_{d, N}$,
Then
$$
|\mu_{\bsj,\bsk}(D_{\S_{d,N}})| \lesssim  \begin{cases}
    2^{t/2}\,N^{-1} 2^{-|\bsj|_1} & \text{if } \bsk \in \bbA_\bsj, \\
    2^{-2|\bsj|_1} & \text{if } \bsk \in \bbD_\bsj \setminus \bbA_\bsj,
\end{cases}
$$
where the second case occurs for $\#(\bbD_\bsj \setminus \bbA_\bsj) \geq 2^{|\bsj|_1} - N$ indices $\bsk$.
%At least $2^{|\bsj|_1}-N$ such intervals contain no points of $\S_{d,N}$ and in such cases we have 
%\[|\mu_{\bsj,\bsk}(D_{\S_{d,N}})| \lesssim 2^{-2|\bsj|_1}.\]
\item On the other hand, if $|\bsj|_1+t/2< \lfloor \ld N \rfloor$, so that $n_\gamma \le |\bsj|_1+t/2< n_{\gamma+1}$ for some $\gamma \in \{0, 1, \dots, u-1\}$, where $N=2^{n_u}+\cdots+2^{n_1}$ with $n_u > \cdots >n_1 \ge 0$ and also setting $n_0 \coloneqq 0$, then 
\[
|\mu_{\bsj,\bsk}(D_{\S_{d,N}})| \lesssim 2^t\,N^{-1} \left(2^{-|\bsj|_1} +  (2n_{\gamma+1} - t - 2|\bsj|_1)^{d-1}2^{-n_{\gamma+1}}\right).
\]
\end{enumerate}
\end{lem}

\begin{proof}
This is \cite[Equations (13), (14), and (16)]{DHMP16}. The statement in the present form can also be found in \cite[Lemma~4.5]{DHMP}.
\end{proof}

\section{Upper bound for order-2 digital sequences}\label{sec:main}

For order-$2$ digital $(t,m,d)$-nets $\P_{d,2^m}$ over $\FF_2$ we know from \cite[Theorem~5.3]{HMOU} that for $p,q \in [1,\infty]$ and $r \in (1/p,2)$ we have 
$$
\wce(\P_{d,2^m}, S^r_{p, q}B(\bbT^d)) \lesssim_{p,q,r,d} 2^{rt/2} \frac{(\log N)^{(d-1)(1-1/q)}}{N^r},\qquad \mbox{where}\ N=2^m.
$$
For an infinite sequence we then have the following result, which constitutes the main result of this work. 

\begin{thm}\label{thm1}
Let $\S_d$ be an order-2 digital $(t,d)$-sequence over $\mathbb{F}_2$ with generating matrices $E_i=(e_{i,k,\ell})_{k,\ell \ge 1}$ for which $e_{i,k,\ell}=0$ for all $k > 2 \ell$ and for all $i \in [d]$. For $N \in \N$ let $\S_{d,N}$ denote the point set consisting of the $N$ initial terms of $\S_d$. Then for $p,q \in [1,\infty]$ and $1/p < r <2$ we have for all $N \in \bbN\setminus\{1\}$, 
\begin{equation*}
\wce(\S_{d,N}, S^r_{p, q}B(\bbT^d)) \lesssim_{p,q,r,d} 2^{(r+1-1/p)t/2} \times \left\{ 
\begin{array}{ll}
(\log N)^{(d-1)(1-1/q)}N^{-\min(r,1)} & \mbox{if } r \not= 1,\\[0.5em]
(\log N)^{d(1-1/q)}N^{-1} & \mbox{if } r = 1.%\\[0.5em]
%\frac{(\log N)^{(d-1)(1-1/q)}}{N} & \mbox{if } r>1.
\end{array}\right. 
\end{equation*}
\end{thm}

The proof of this result will be given in Section~\ref{sec:proof}. But first, the result should be classified in comparison to lower bounds and a corollary regarding our initial problem measure, the periodic $L_2$-discrepancy, should be formulated.

\begin{rem}\rm 
It is known that for $p,q \in [1,\infty]$ and any $r>1/p$ we have for every $N$-element set $\calP_{d,N}$ in $[0,1)^d$, 
$$
\wce(\P_{d,N}, S^r_{p, q}B(\bbT^d)) \gtrsim \frac{(\log N)^{(d-1)(1-1/q)}}{N^r}.
$$ This lower bound even holds for general cubature rules, not only for QMC rules, see \cite[Theorem~4.1]{DU15} or \cite{T90} for the case of Sobolev spaces; the result is also stated in \cite[Theorem~5.5]{HMOU}. This shows that the order of magnitude in $N$ in Theorem~\ref{thm1} is best possible if $r \in (1/p,1)$, which is non-empty for $p>1$. For $r\ge 1$ the upper bound in Theorem~\ref{thm1} is, apart from the logarithm factors, of order $\mathcal{O}(1/N)$. Considering infinite sequences, also this order of magnitude cannot be improved for all $N$ and not even for infinitely many $N$ in an arithmetic progression. Indeed, the following result, based on an argument by Sobol' \cite{S98}, was stated in \cite[Proposition~1]{HKKN12}. If for some $M \in \bbN$, any $M$-point QMC rule has a positive worst-case error, which is bounded below away from zero, i.e.,
\begin{equation*}%\label{eq:assump_Sobol}
\inf_{\substack{\P_{d,M} \subseteq [0,1)^d\\ \#\P_{d,M}=M}}  \wce(\P_{d,M}, S^r_{p, q}B(\bbT^d)) \ge c_M > 0,
\end{equation*}
where $c_M$ is allowed to depend on $M$, then the sequence 
$(\wce(\S_{d,N_0+n M}, S^r_{p, q}B(\bbT^d)))_{n\ge 0}$
cannot converge to zero faster than $\mathcal{O}(1/N)$, with $N$ of the form $N = N_0 + n M$, for any fixed $N_0 \in \bbN$. It is easy to see that the required assumption is true for any $M\in \bbN$ in the space $S^r_{p, q}B(\bbT^d)$. Thus, the order of magnitude in Theorem~\ref{thm1} is also optimal in the case $r \in [1,2)$ (up to logarithmic factors). For $r=1$ and $p=q=2$ the obtained order of magnitude is optimal even with regard to logarithmic factors (see Corollary~\ref{cor:perL2} below).
\end{rem}

By \eqref{eq:wce_lper}, if in Theorem~\ref{thm1} the parameters $p,q$ are set to 2 and $r$ to 1, then we obtain the desired best possible estimate for the periodic $L_2$-discrepancy/diaphony of order-2 digital sequences.

\begin{cor}\label{cor:perL2}
Let $\S_d$ be an order-2 digital $(t,d)$-sequence over $\mathbb{F}_2$ with generating matrices $E_i=(e_{i,k,\ell})_{k,\ell \ge 1}$ for which $e_{i,k,\ell}=0$ for all $k > 2 \ell$ and for all $i \in [d]$. Then we have
\begin{equation*}
L_{2,N}^{{\rm per}}(\S_d) \lesssim_d 2^t \ \frac{(\log N)^{d/2}}{N} \qquad \mbox{for all } N \in \N \setminus\{1\}, 
\end{equation*}
and this is best possible according to Pro{\u\i}nov's lower bound~\eqref{dia:lbd}.
\end{cor}

\begin{rem}\rm
It is well known that order-2 digital sequences exhibit the best possible order for the usual star (anchored) $L_2$-discrepancy and even for the star $L_p$-discrepancy for any $p \in (1,\infty)$; see \cite{DHMP16}. The same holds true for the extreme (unanchored) $L_p$-discrepancy. In this case, the upper bound results from the aforementioned bound for the star $L_p$-discrepancy from \cite{DHMP16} together with the fact that the extreme $L_p$-discrepancy is dominated---up to a multiplicative factor that depends only on $p$ and $d$---by the star $L_p$-discrepancy (see \cite[Corollary~5]{KP22b}), and the lower bound from \cite[Theorem~1]{KP22a}. Furthermore, order-2 digital sequences exhibit optimal rates for star discrepancy with respect to the exponential Orlicz norm, 
the BMO semi-norm, as well as the Sobolev, Besov, and Triebel-Lizorkin norms with dominating mixed smoothness, see \cite[Theorems~2.1, 2.2, 3.1 and 3.2]{DHMP}. The present Corollary~\ref{cor:perL2} guarantees the same behavior for the periodic $L_2$-discrepancy. This is further evidence of the universally excellent distribution properties of Dick's order-2 digital sequences with respect to various discrepancies and figures of merit.
\end{rem}

\section{The proof of Theorem~\ref{thm1}}\label{sec:proof}

Before we begin proving Theorem~\ref{thm1}, we gather some further auxiliary results.

\begin{lem}\label{le:partS}
Let $\S_d$ be an order-$2$ digital $(t,d)$-sequence over $\mathbb{F}_2$ with generating matrices $C_1,\ldots,C_d$, with $C_j = (c_{j,k,\ell})_{k, \ell \ge 1}$ for which $c_{i,k,\ell} = 0$ for all $k > 2 \ell$. Let $N \in \N$ with dyadic expansion $N=2^{n_u}+\cdots+2^{n_1}$, $n_u>\cdots>n_1\geq 0$. For $\gamma \in \{1, \ldots, u\}$ let
\begin{equation*}
\P_{d,2^{n_\gamma}} \coloneqq \{\boldsymbol{x}_{2^{n_1}+\cdots + 2^{n_{\gamma-1}}}, \boldsymbol{x}_{2^{n_1}+\cdots + 2^{n_{\gamma-1}} + 1}, \ldots, \boldsymbol{x}_{-1+2^{n_1} + \cdots + 2^{n_\gamma}}\},
\end{equation*}
where for $\gamma = 1$ we set $2^{n_1} + \cdots + 2^{n_{\gamma-1}} = 0$. Then the point set $\S_{d,N}$ consisting of the first $N$ elements of the sequence $\S_d$ is a union of $\P_{d,2^{n_\gamma}}$ for $\gamma \in \{1, \ldots, u\}$ and each $\P_{d,2^{n_\gamma}}$ is a digitally shifted order-$2$ digital $(t,n_\gamma,d)$-net over $\mathbb{F}_2$ with generating matrices $C_{1,2 n_\gamma \times n_\gamma},\ldots,C_{d,2 n_\gamma \times n_\gamma}$, i.e., the left upper $2 n_\gamma \times n_\gamma$-submatrices of $C_1,\ldots,C_d$. Furthermore, every shift vector only has a finite number of components different from zero.     
\end{lem}

\begin{proof}
The result is the first claim in \cite[Proof of Theorem~2.2]{DHMP16}. See there for a proof.   
\end{proof}

Finally, we need the following technical but elementary results. 

\begin{lem}\label{index_dim_red}
Let $r\in\N_0$ and $s\in\N$. Then
\[ \#\{(a_1,\ldots,a_s)\in\N_0^s:\; a_1 + \cdots + a_s = r\} \leq (r + 1)^{s-1}. \]
\end{lem}

\begin{proof}
See \cite[Proof of Lemma~16.26]{DP10}.
\end{proof}

\begin{lem}\label{index_dim_red_log}
Let $K > 0$, $A>1$, and $q,s \ge 0$. Then we have
\[ \sum_{\substack{r \in \bbN_0 \\ r < K}} A^r (K-r)^q r^s \lesssim A^K\,K^s, \]
 where the implicit constant is independent of $K$.
\end{lem}

\begin{proof}
The statement is a simple modification of \cite[Lemma~5.2]{M15} with essentially the same proof.
\end{proof}

\begin{lem}\label{le4}
For $\alpha >0$, $d, t \in \bbN$ and $N \in \bbN \setminus \{1\}$ we have
$$\sum_{\bsj \in \bbN_{0}^d \atop |\bsj|_1+t/2\geq \lfloor\ld N\rfloor} 2^{- \alpha |\bsj|_1} \lesssim  2^{\alpha t/2} \frac{(\log N)^{d-1}}{N^{\alpha}},$$
with an implicit constant independent of $N$ or $t$.
\end{lem}

\begin{proof}
We have, with Lemma~\ref{index_dim_red},
\begin{eqnarray*}
\sum_{\bsj \in \bbN_{0}^d \atop |\bsj|_1+t/2\geq \lfloor\ld N\rfloor} 2^{- \alpha |\bsj|_1}  & = & \sum_{\ell \ge \lfloor\ld N\rfloor - t/2} \frac{1}{2^{\ell \alpha}} \sum_{\bsj \in \bbN_0 \atop |\bsj|_1=\ell} 1 \\
& \le & \sum_{\ell \ge \lfloor\ld N\rfloor - t/2} \frac{(\ell+1)^{d-1}}{2^{\ell \alpha}} \lesssim  2^{\alpha t/2} \frac{(\log N)^{d-1}}{N^{\alpha}}.
\end{eqnarray*}
\end{proof}

The following lemma is similar to \cite[Lemma~13.24]{DP10}.

\begin{lem} \label{lem:A_1/2}
    Let $A \geq 1/2$, $s \in \bbR$ and $b > 1$. Then
    $$
    \sum_{k=0}^\infty \frac{(A+k)^s}{b^k} \lesssim A^s
    $$
    with an implicit constant independent of $A$.
\end{lem}

\begin{proof}
    Since $A \geq 1/2$ it holds \peter{that} $A+k \leq (2k+1) A$ for all $k \in \bbN_0$, and consequently
    $$
    \sum_{k=0}^\infty \frac{(A+k)^s}{b^k} \leq \left(\sum_{k=0}^\infty \frac{(2k+1)^s}{b^k}\right) A^s.
    $$
    The infinite sum converges to a finite value, proving the claim.
\end{proof}

We are now in the position to prove our main result, Theorem~\ref{thm1}.

\begin{proof}[Proof of Theorem~\ref{thm1}]
We consider only the case $p,q>1$. For $p=1$ or $q=1$, the result is obtained in the same way after the usual adjustments. Let $\S_d$ be an order-$2$ digital $(t,d)$-sequence over $\mathbb{F}_2$ with generating matrices $C_1,\ldots,C_d$, with $C_j = (c_{j,k,\ell})_{k, \ell \ge 1}$ for which $c_{\peter{j},k,\ell} = 0$ for all $k > 2 \ell$. Let $N \in \N$ with dyadic expansion $N=2^{n_u}+\cdots+2^{n_1}$ with $n_u>\cdots>n_1\geq 0$. 

Using Lemma~\ref{est:err:net} we need to analyze the right-hand side of the inequality
\begin{equation*}
\wce(\S_{d,N}, S^r_{p, q}B(\bbT^d))^{q'} \lesssim \sum_{\bsj \in \bbN_{-1}^d} 2^{-|\bsj|_1 (r-1/p) q'} \left(\sum_{\bsk \in \bbD_{\bsj}} |c_{\bsj, \bsk}(\S_{d,N})|^{p'}\right)^{q'/p'}.
\end{equation*}
We first deal with the partial sum
\begin{eqnarray*}
\Sigma_1 & \coloneqq &\sum_{\bsj \in \bbN_{0}^d} 2^{-|\bsj|_1 (r-1/p) q'} \left(\sum_{\bsk \in \bbD_{\bsj}} |c_{\bsj, \bsk}(\S_{d,N})|^{p'}\right)^{q'/p'}\\
& = & \sum_{\bsj \in \bbN_{0}^d} 2^{-|\bsj|_1 (r-1-1/p) q'} \left(\sum_{\bsk \in \bbD_{\bsj}} |\mu_{\bsj, \bsk}(D_{\S_{d,N}})|^{p'}\right)^{q'/p'},
\end{eqnarray*}
where we have applied Lemma~\ref{le1}.

Now we split up the latter sum according to the size $|\bsj|_1$ of $\bsj \in \bbN_0^d$. Let us first consider the case where $|\bsj|_1+t/2\geq \lfloor \ld N \rfloor$. We have
\begin{eqnarray*}
\Sigma_{1,1} & \coloneqq & \sum_{\bsj \in \bbN_{0}^d \atop |\bsj|_1+t/2\geq \lfloor \ld N \rfloor} 2^{-|\bsj|_1 (r-1-1/p) q'} \left(\sum_{\bsk \in \bbD_{\bsj}} |\mu_{\bsj, \bsk}(D_{\S_{d,N}})|^{p'}\right)^{q'/p'} \\    
& \lesssim & \sum_{\bsj \in \bbN_{0}^d \atop |\bsj|_1+t/2\geq \lfloor \ld N \rfloor} 2^{-|\bsj|_1 (r-1-1/p) q'} \left(\sum_{\bsk \in \bbA_{\bsj}} |\mu_{\bsj, \bsk}(D_{\S_{d,N}})|^{p'}\right)^{q'/p'}\\
& & + \sum_{\bsj \in \bbN_{0}^d \atop |\bsj|_1+t/2\geq \lfloor \ld N \rfloor} 2^{-|\bsj|_1 (r-1-1/p) q'} \left(\sum_{\bsk \in \bbD_{\bsj}\setminus \bbA_{\bsj}} |\mu_{\bsj, \bsk}(D_{\S_{d,N}})|^{p'}\right)^{q'/p'},
\end{eqnarray*}
where, like in Lemma~\ref{le131416}, $\bbA_{\bsj}$ denotes the set of indices $\bsk \in \bbD_{\bsj}$ for which $I_{\bsj,\bsk}$ contains a point of $\S_{d,N}$. Clearly $\# \bbA_{\bsj} \le N$ and $\#(\bbD_\bsj \setminus \bbA_\bsj) \leq 2^{|\bsj|_1}$. We now apply Lemma~\ref{le131416}, which yields
\begin{eqnarray*}
\Sigma_{1,1} & \lesssim & \sum_{\bsj \in \bbN_{0}^d \atop |\bsj|_1+t/2\geq \lfloor \ld N \rfloor} 2^{-|\bsj|_1 (r-1-1/p) q'} \left(N \frac{1}{N^{p'}} \frac{2^{p' t/2}}{2^{p'|\bsj|_1}}\right)^{q'/p'}\\
& & + \sum_{\bsj \in \bbN_{0}^d \atop |\bsj|_1+t/2\geq \lfloor \ld N \rfloor} 2^{-|\bsj|_1 (r-1-1/p) q'} \left( \frac{2^{|\bsj|_1}}{2^{2 p' |\bsj|_1}}  \right)^{q'/p'}\\
& = &  N^{q'/p'-q'} 2^{q' t/2} \sum_{\bsj \in \bbN_{0}^d \atop |\bsj|_1+t/2\geq \lfloor \ld N \rfloor} 2^{-|\bsj|_1 (r-1/p) q'} + \sum_{\bsj \in \bbN_{0}^d \atop |\bsj|_1+t/2\geq \lfloor \ld N \rfloor} 2^{-|\bsj|_1 r q'}. 
\end{eqnarray*}
Employing Lemma~\ref{le4}, we obtain 
\begin{eqnarray*}
\Sigma_{1,1} 
& \lesssim &  N^{q'/p'-q'} 2^{q' t/2} 2^{(r-1/p) q' t/2} \frac{(\log N)^{d-1}}{N^{(r-1/p) q'}} + 2^{r q' t/2} \frac{(\log N)^{d-1}}{N^{r q'}}\\  
& = & 2^{q'(r+1-1/p) t/2}  \frac{(\log N)^{d-1}}{N^{r q'}} + 2^{q' r t/2} \frac{(\log N)^{d-1}}{N^{r q'}}\\
& \lesssim & 2^{q' (r+1-1/p) t/2} \frac{(\log N)^{d-1}}{N^{r q'}},
\end{eqnarray*}
where we also used $r \leq r+1-1/p$.

We now turn to the more demanding case of $|\bsj|_1+t/2<\lfloor \ld N \rfloor$. Like in Lemma~\ref{le:partS} let $\S_{d,N}= \bigcup_{\gamma=1}^u \P_{d,2^{n_{\gamma}}}$, where each $\P_{d,2^{n_{\gamma}}}$ is a digitally shifted order-2 digital $(t,n_\gamma,d)$-net over $\mathbb{F}_2$. \peter{Now, fix} $\gamma\in \{0, 1,\ldots,u-1\}$ with
\[ 
n_{\gamma} \le |\bsj|_1+t/2 < n_{\gamma+1},
\]
where we set $n_0 \coloneqq 0$ (note that $n_u = \lfloor \ld N\rfloor$). %and $n_{u+1} \coloneqq \lfloor\ld N\rfloor $\nicolas{[But then $n_{u+1}=n_u$ so the case $n_u \leq ... < n_{u+1}$ does not occur]}.
We put $M\coloneqq 2^{n_\gamma}+\cdots+2^{n_1}$ for $\gamma > 0$ and $M\coloneqq0$ for $\gamma=0$. Furthermore, we write $\widetilde{\P}_{d,M} = \bigcup_{i=1}^{\gamma} \P_{d,2^{n_i}}$; in particular $\widetilde{\P}_{d,0} = \emptyset$. Note that we then have 
$ \S_{d,N}=\widetilde{\P}_{d,M} \cup \P_{d,2^{n_{\gamma+1}}} \cup\ldots\cup \P_{d,2^{n_{u}}}$.
\begin{itemize}
\item If $\gamma=0$, we use the estimate from Lemma~\ref{haarcoeffdignets} to get
\begin{equation}\label{eq:lem_disc_net_tvalue}
|\mu_{\bsj,\bsk}(D_{\P_{d,2^{n_\kappa}}})| \lesssim  \frac{(2n_\kappa - t - 2|\bsj|_1)^{d-1}}{2^{2n_\kappa - t}}, 
\end{equation}
for $\kappa\in \{1,\ldots , u\}$.
\item If $1\le \gamma < u$, we again use \eqref{eq:lem_disc_net_tvalue} for $\kappa \in \{\gamma+1,\ldots , u\}$, combined with Item~1 in Lemma~\ref{le131416}, from which we obtain
\begin{equation*}%\label{eq:lem_disc_net_131416}
|\mu_{\bsj,\bsk}(D_{\widetilde{\P}_{d,M}})| \lesssim  \frac{1}{M} \frac{1}{2^{|\bsj|_1-t/2}}, 
\end{equation*}
since $|\bsj|_1+t/2\ge \lfloor \ld M \rfloor$ and thus also $2^{-2|\bsj|_1} \lesssim M^{-1} 2^{-|\bsj|_1+t/2}$. 
%\item If $\mu=u$, we have $\S_{d,N}=\widetilde{\P}_{d,M}$, and can just use the bound~\eqref{eq:lem_disc_net_131416}.
\end{itemize}

Together with the linearity of the local discrepancy function and the triangle inequality this leads to 
\begin{align*}%\label{asdf}
 |\mu_{\bsj,\bsk}(D_{\S_{d,N}})| &\le \frac{M}{N} |\mu_{\bsj,\bsk}(D_{\widetilde{\P}_{d,M}})| + \frac{1}{N}\sum_{\kappa=\gamma+1}^u 2^{n_\kappa} |\mu_{\bsj,\bsk}(D_{\P_{d,2^{n_\kappa}}})|\nonumber\\
 &\lesssim \frac{1}{N}\left( \frac{2^{t/2}}{2^{|\bsj|_1}} + \sum_{\kappa=\gamma+1}^u  \frac{(2n_\kappa - t - 2|\bsj|_1)^{d-1}}{2^{n_\kappa-t}}\right)\nonumber\\
 &\lesssim \frac{1}{N} \left(\frac{2^{t/2}}{2^{|\bsj|_1}} + 2^t\sum_{k=0}^{\infty}  \frac{(2n_{\gamma+1} +2k - t - 2|\bsj|_1)^{d-1}}{2^{n_{\gamma+1}+k}}\right)\nonumber\\
 &\lesssim \frac{1}{N} \left(\frac{2^{t/2}}{2^{|\bsj|_1}} + 2^t \frac{(2n_{\gamma+1} - t - 2|\bsj|_1)^{d-1}}{2^{n_{\gamma+1}}}\right),
\end{align*}
where we used Lemma~\ref{lem:A_1/2} %\cite[Lemma~13.24]{DP10}
in the last step with $A = n_{\gamma+1}-t/2-|\bsj|_1 \geq 1/2$. %Note that the final estimate does not depend on $\bsk \in \bbD_\bsj$.}
We now use this bound to estimate the terms of $\Sigma_1$ for which $|\bsj|_1+t/2 < \lfloor\ld N\rfloor$. We have
\begin{eqnarray*}
\Sigma_{1,2} & \coloneqq & \sum_{\bsj \in \bbN_{0}^d \atop |\bsj|_1+t/2 < \lfloor\ld N\rfloor} 2^{-|\bsj|_1 (r-1-1/p) q'} \left(\sum_{\bsk \in \bbD_{\bsj}} |\mu_{\bsj, \bsk}(D_{\S_{d,N}})|^{p'}\right)^{q'/p'} \\ 
& = & \sum_{\gamma=0}^{u-1} \sum_{n_{\gamma} \le |\bsj|_1+t/2 < n_{\gamma+1}}  2^{-|\bsj|_1 (r-1-1/p) q'} \left(\sum_{\bsk \in \bbD_{\bsj}} |\mu_{\bsj, \bsk}(D_{\S_{d,N}})|^{p'}\right)^{q'/p'}\\
& \lesssim & \sum_{\gamma=0}^{u-1} \sum_{n_{\gamma} \le |\bsj|_1+t/2 < n_{\gamma+1}}  2^{-|\bsj|_1 (r-1-1/p) q'} 2^{|\bsj|_1 q'/p'} \frac{1}{N^{q'}} \left(\frac{2^{t/2}}{2^{|\bsj|_1}} + 2^t \frac{(2n_{\gamma+1} - t - 2|\bsj|_1)^{d-1}}{2^{n_{\gamma+1}}}\right)^{q'} \\
& = & \frac{1}{N^{q'}} \sum_{\gamma=0}^{u-1} \sum_{n_{\gamma} \le |\bsj|_1+t/2 < n_{\gamma+1}}  2^{-|\bsj|_1 (r-2) q'} \left(\frac{2^{t/2}}{2^{|\bsj|_1}} + 2^t \frac{(2n_{\gamma+1} - t - 2|\bsj|_1)^{d-1}}{2^{n_{\gamma+1}}}\right)^{q'}.
\end{eqnarray*}
For any $\beta\geq1$ the function $f(x)=x^{\beta}$ is convex, i.e. it holds \peter{that} $f(\lambda x +(1-\lambda) y) \le \lambda f(x)+(1-\lambda) f(y)$ for $\lambda \in [0,1]$. With $\lambda=\tfrac{1}{2}$ we obtain $$\left(\frac{x+y}{2}\right)^{\beta} =f(\tfrac{1}{2} x+ \tfrac{1}{2} y) \le \tfrac{1}{2} (f(x)+f(y)) = \frac{x^{\beta}+y^{\beta}}{2}$$ and thus $(x+y)^{\beta} \le 2^{\beta-1} (x^{\beta}+y^{\beta})$ for all $\beta \ge 1$. Hence, since $q' \ge 1$, we obtain
\begin{eqnarray*}
\Sigma_{1,2} & \lesssim & \frac{1}{N^{q'}} \sum_{\gamma=0}^{u-1} \sum_{n_{\gamma} \le |\bsj|_1+t/2 < n_{\gamma+1}}  2^{-|\bsj|_1 (r-2) q'} \left(\frac{2^{q't/2}}{2^{q' |\bsj|_1}} + 2^{q't} \frac{(2n_{\gamma+1} - t - 2|\bsj|_1)^{q'(d-1)}}{2^{q' n_{\gamma+1}}}\right)\\
& = & \frac{2^{q' t/2}}{N^{q'}} \sum_{\gamma=0}^{u-1} \sum_{n_{\gamma} \le |\bsj|_1+t/2 < n_{\gamma+1}}  2^{|\bsj|_1 (1-r) q'} \\
& & + \frac{2^{q' t}}{N^{q'}} \sum_{\gamma=0}^{u-1} \frac{1}{2^{q' n_{\gamma+1}}} \sum_{n_{\gamma} \le |\bsj|_1+t/2 < n_{\gamma+1}}  2^{|\bsj|_1 (2-r) q'} (2n_{\gamma+1} - t - 2|\bsj|_1)^{q'(d-1)}.
\end{eqnarray*}
For the first sum we have, using Lemma~\ref{index_dim_red},
\begin{eqnarray*}
\sum_{\gamma=0}^{u-1} \sum_{n_{\gamma} \le |\bsj|_1+t/2 < n_{\gamma+1}}  2^{|\bsj|_1 (1-r) q'} & \le &  \sum_{\gamma=0}^{u-1} \sum_{n_{\gamma} \le \ell+t/2 < n_{\gamma+1}} 2^{\ell(1-r) q'} (\ell+1)^{d-1}\\
& \lesssim & (\log N)^{d-1} \sum_{\ell < \lfloor\ld N\rfloor}  2^{\ell(1-r) q'}\\
& \lesssim & \left\{
\begin{array}{ll}
(\log N)^{d-1} N^{(1-r) q'} & \mbox{if } r<1,\\
(\log N)^d & \mbox{if } r=1,\\
(\log N)^{d-1} & \mbox{if } r>1.
\end{array}\right.
\end{eqnarray*}
For the second sum we have %\nicolas{[Second to last line may be omitted, covered by the ``$r^s$'' portion of Lem. 18]}
\begin{eqnarray*}
\lefteqn{\sum_{\gamma=0}^{u-1} \frac{1}{2^{q' n_{\gamma+1}}} \sum_{n_{\gamma} \le |\bsj|_1+t/2 < n_{\gamma+1}}  2^{|\bsj|_1 (2-r) q'} (2n_{\gamma+1} - t - 2|\bsj|_1)^{q'(d-1)}}\\
& \lesssim & \sum_{\gamma=0}^{u-1} \frac{1}{2^{q' n_{\gamma+1}}} \sum_{n_{\gamma} \le \ell+t/2 < n_{\gamma+1}}  2^{\ell (2-r) q'} (2n_{\gamma+1} - t - 2 \ell)^{q'(d-1)} (\ell+1)^{d-1}\\
%& \lesssim & (\log N)^{d-1} \sum_{\mu=0}^u \frac{1}{2^{q' n_{\mu+1}}} \sum_{n_{\mu} \le \ell+t/2 < n_{\mu+1}}  2^{\ell (2-r) q'} (2n_{\mu+1} - t - 2 \ell)^{q'(d-1)}\\
& \lesssim & (\log N)^{d-1}  \sum_{\gamma=0}^{u-1} \frac{1}{2^{q' n_{\gamma+1}}} 2^{(2-r)q'(n_{\gamma+1}-t/2)}, 
\end{eqnarray*}
where we used Lemmas~\ref{index_dim_red} and \ref{index_dim_red_log}. Hence%\nicolas{, since $2^{-(2-r) t/2} \leq 1$ by $r < 2$,}
%\begin{eqnarray*}
%\lefteqn{\sum_{\mu=0}^u \frac{1}{2^{q' n_{\mu+1}}} \sum_{n_{\mu} \le |\bsj|_1+t/2 < n_{\mu+1}}  2^{|\bsj|_1 (2-r) q'} (2n_{\mu+1} - t - 2|\bsj|_1)^{q'(d-1)}}\\
%& \lesssim & (\log N)^{d-1}  \sum_{\mu=0}^u 2^{(1-r)q'n_{\mu+1}}\\
%& \lesssim & \left\{
%\begin{array}{ll}
%(\log N)^{d-1} N^{(1-r) q'} & \mbox{if } r<1,\\
%(\log N)^d & \mbox{if } r=1,\\
%(\log N)^{d-1} & \mbox{if } r>1.
%\end{array}\right. 
%\end{eqnarray*}
\begin{align*}
    & \sum_{\gamma=0}^{u-1} \frac{1}{2^{q' n_{\gamma+1}}} \sum_{n_{\gamma} \le |\bsj|_1+t/2 < n_{\gamma+1}}  2^{|\bsj|_1 (2-r) q'} (2n_{\gamma+1} - t - 2|\bsj|_1)^{q'(d-1)}\\
    & ~~ \lesssim ~~ 2^{-(2-r)q' t/2} (\log N)^{d-1}  \sum_{\gamma=0}^{u-1} 2^{(1-r)q'n_{\gamma+1}}\\
    & ~~ \lesssim ~~ 2^{-(2-r)q' t/2} \times \left\{
\begin{array}{ll}
(\log N)^{d-1} N^{(1-r) q'} & \mbox{if } r<1,\\
(\log N)^d & \mbox{if } r=1,\\
(\log N)^{d-1} & \mbox{if } r>1.
\end{array}\right. 
\end{align*}
Altogether, we obtain 
\begin{eqnarray*}
\Sigma_{1,2} & \lesssim & \frac{2^{\max(1, r) q' t/2}}{N^{q'}} \times \left\{
\begin{array}{ll}
(\log N)^{d-1} N^{(1-r) q'} & \mbox{if } r < 1,\\
(\log N)^d & \mbox{if } r=1,\\
(\log N)^{d-1} & \mbox{if } r>1.
\end{array}\right. 
\end{eqnarray*}
Adding $\Sigma_{1,1}$ and $\Sigma_{1,2}$, noting that $1 < r+1-1/p$ and $r \leq r+1-1/p$, we obtain
\begin{equation}\label{est:S1}
\Sigma_1 =\Sigma_{1,1}+\Sigma_{1,2} \lesssim 2^{q'(r+1-1/p)t/2} \times \left\{ 
\begin{array}{ll}
\frac{(\log N)^{d-1}}{N^{rq'}} & \mbox{if } r < 1,\\[0.5em]
\frac{(\log N)^d}{N^{q'}} & \mbox{if } r=1,\\[0.5em]
\frac{(\log N)^{d-1}}{N^{q'}} & \mbox{if } r>1.
\end{array}\right. 
\end{equation}

Finally, we have to consider the partial sum
\begin{eqnarray*}
\Sigma_2 & \coloneqq & \sum_{\bsj \in \bbN_{-1}^d \setminus \bbN_0^d}2^{-|\bsj|_1 (r-1/p) q'} \left(\sum_{\bsk \in \bbD_{\bsj}} |c_{\bsj, \bsk}(\S_{d,N})|^{p'}\right)^{q'/p'}\\
& = & \sum_{\uu \subsetneq [d]} \sum_{\bsj \in \mathcal{N}_{\uu}} 2^{-|\bsj|_1 (r-1/p) q'} \left(\sum_{\bsk \in \bbD_{\bsj}} |c_{\bsj, \bsk}(\S_{d,N})|^{p'}\right)^{q'/p'} 
\end{eqnarray*}
where for $\uu \subseteq [d]$ we put $\mathcal{N}_{\uu}\coloneqq\{\bsj \in \bbN_{-1}^d : j_i \in \bbN_0 \mbox{ if } i \in \uu, \mbox{ and } j_i=-1 \mbox{ if } i \in [d]\setminus\uu\}$. Note that the projection of an order-2 digital $(t,d)$-sequence over $\bbF_2$ onto coordinates in $\uu \subseteq [d]$ is an order-2 digital $(t,|\uu|)$-sequence over $\bbF_2$. Hence, using Lemma~\ref{le1} and \eqref{est:S1} with $d$ replaced by $|\uu|$, we obtain
\begin{equation*}
\sum_{\bsj \in \mathcal{N}_{\uu}} 2^{-|\bsj|_1 (r-1/p) q'} \left(\sum_{\bsk \in \bbD_{\bsj}} |c_{\bsj, \bsk}(\S_{d,N})|^{p'}\right)^{q'/p'}
\lesssim 2^{q'(r+1-1/p)t/2} \times \left\{ 
\begin{array}{ll}
\frac{(\log N)^{|\uu|-1}}{N^{rq'}} & \mbox{if } r < 1,\\[0.5em]
\frac{(\log N)^{|\uu|}}{N^{q'}} & \mbox{if } r = 1,\\[0.5em]
\frac{(\log N)^{|\uu|-1}}{N^{q'}} & \mbox{if } r>1.
\end{array}\right. 
\end{equation*}
Note that for $\uu=\emptyset$ we obtain $c_{(-1,\ldots,-1),(0,\ldots,0)}(\S_{d,N})=0$. Hence, for $r \ge 1/p$, 
\begin{equation*}
\Sigma_2 \lesssim 2^{q'(r+1-1/p)t/2} \times \left\{ 
\begin{array}{ll}
\frac{(\log N)^{d-2}}{N^{rq'}} & \mbox{if } r < 1,\\[0.5em]
\frac{(\log N)^{d-1}}{N^{q'}} & \mbox{if } r = 1,\\[0.5em]
\frac{(\log N)^{d-2}}{N^{q'}} & \mbox{if } r>1.
\end{array}\right. 
\end{equation*}
Finally,
\begin{equation*}
\wce(\S_{d,N}, S^r_{p, q}B(\bbT^d))^{q'} = \Sigma_1 + \Sigma_2 \lesssim 2^{q'(r+1-1/p)t/2} \times \left\{ 
\begin{array}{ll}
\frac{(\log N)^{d-1}}{N^{rq'}} & \mbox{if } r < 1,\\[0.5em]
\frac{(\log N)^{d}}{N^{q'}} & \mbox{if } r = 1,\\[0.5em]
\frac{(\log N)^{d-1}}{N^{q'}} & \mbox{if } r>1,
\end{array}\right.
\end{equation*}
and the result is obtained by taking the $q'$-th root, recalling that $1/q' = 1 - 1/q$.
\end{proof}

\section*{Acknowledgment}

The first and second authors acknowledge the support of the Austrian Science Fund (FWF) Project  P 34808/Grant DOI: 10.55776/P34808. For open access purposes, the authors have applied a CC BY public copyright license to any author accepted manuscript version arising from this submission.

\noindent{\bf Authors' Addresses:}

\noindent Peter Kritzer, Johann Radon Institute for Computational and Applied Mathematics, Austrian Academy of Sciences, Altenbergerstra{\ss}e 69, 4040 Linz, Austria.\\ Email: peter.kritzer(AT)oeaw.ac.at\\ ORCiD: 0000-0002-7919-7672\\

\noindent Nicolas Nagel, Johann Radon Institute for Computational and Applied Mathematics, Austrian Academy of Sciences, Altenbergerstra{\ss}e 69, 4040 Linz, Austria.\\ Email: nicolas.nagel(AT)ricam.oeaw.ac.at\\ ORCiD: 0009-0004-3362-3543\\

\noindent Friedrich Pillichshammer, Institut f\"{u}r Finanzmathematik und Angewandte Zahlentheorie, JKU Linz, Altenbergerstra{\ss}e 69, 4040 Linz, Austria.\\ Email: friedrich.pillichshammer(AT)jku.at\\ ORCiD: 0000-0001-6952-9218

\end{document}